\documentclass[12pt]{amsart}
\usepackage{amsmath,amsthm,amsfonts}
\usepackage{amssymb}
\usepackage[mathscr]{euscript}
\usepackage{longtable}
\usepackage{caption}
\usepackage[titletoc]{appendix}
\usepackage{xcolor,url}
\usepackage{comment}
\usepackage{enumitem}
\usepackage{hyperref}

\newtheorem{theorem}{Theorem}
\newtheorem{lemma}[theorem]{Lemma}
\newtheorem{corollary}[theorem]{Corollary}

\allowdisplaybreaks[1]

\newcommand{\ZZ}{\mathbb{Z}}

\title{Elated Numbers}

 \author{N. Bradley Fox}
     \address{Department of Mathematics and Statistics, Austin Peay State University, 601 College St, Clarksville, TN 37044, USA}
     \email{foxb@apsu.edu}

     \author{Nathan H. Fox}
     \address{Department of Quantitative Sciences, Canisius University, 2001 Main St, Buffalo, NY 14208, USA}
     \email{fox42@canisius.edu}

     \author{Helen G. Grundman}
     \address{Department of Mathematics, Bryn Mawr College, 101 N. Merion Ave, Bryn Mawr, PA 19010, USA}
     \email{grundman@brynmawr.edu}

     \author{Rachel Lynn}
     \address{Department of Mathematics and Physics, Schreiner University, 2100 Memorial Blvd, Kerrville, TX 78028, USA}
     \email{rlynn@schreiner.edu}

     \author{Changningphaabi Namoijam}
     \address{Department of Mathematics, Colby College, 5830 Mayflower Hill, Waterville, ME 04901, USA}
     \email{cnamoijam@gmail.com}

     \author{Mary Vanderschoot}
     \address{Department of Mathematics and Computer Science, Wheaton College, 501 College Ave, Wheaton, IL 60187, USA}
     \email{mary.vanderschoot@wheaton.edu}

\begin{document}

\begin{abstract}
For a base $b \geq 2$, the $b$-elated function, $E_{2,b}$, maps a positive integer written in base $b$ to the product of its leading digit and the sum of the squares of its digits. A $b$-elated number is a positive integer that maps to $1$ under 
iteration of $E_{2,b}$.  The height of a $b$-elated number is the number of iterations required to map it to $1$.  
We determine the fixed points and cycles of $E_{2,b}$ and prove a range of results concerning sequences of $b$-elated numbers and $b$-elated numbers of minimal heights.
Although the $b$-elated function is closely related to the $b$-happy function, the behaviors of the two are notably different, as demonstrated by the results in this work.
\end{abstract}

\maketitle

\section{Introduction}
Happy numbers and their variations have been studied for many decades.  (For a recent survey on the subject, see~\cite{GH22,GH22C}.)  In this paper, we consider a function similar to the happy function, but with a key difference.

For $x \in \ZZ^+$, we write the {\em base $b$ expansion of $x$} as $x=\sum_{i=0}^nx_ib^i$ where $0\leq x_i <b$ and the {\em leading digit} $x_n \neq 0$. Recall the standard happy function, generalized to an arbitrary base~\cite{GT01}:  For $b\geq 2$, define $S_{2,b}: \ZZ^+ \rightarrow \ZZ^+$ by
\begin{equation*}
S_{2,b}(a) = S_{2,b}\!\left(\sum_{i=0}^n a_i b^i\right) = \sum_{i=0}^n a_i^2,
\end{equation*}
where $\sum_{i=0}^n a_i b^i$ is the base $b$ expansion of $a$.
Any positive integer that maps to 1 under iteration of $S_{2,b}$ is called a {\em $b$-happy number.}  A $10$-happy number is usually referred to simply as  a {\em happy number.}

We now define, for any base $b\geq 2$, the {\em $b$-elated function}, $E_{2,b}:\ZZ^+ \rightarrow \ZZ^+$, by
\[
E_{2,b}(a) = E_{2,b}\!\left(\sum_{i=0}^n a_i b^i\right) = a_n\sum_{i=0}^n a_i^2,
\]
where, again, $\sum_{i=0}^n a_i b^i$ is the base $b$ expansion of $a$.
Any positive integer that maps to 1 under iteration of $E_{2,b}$  is called a \emph{$b$-elated number}.  More precisely, 
an integer $a \in \ZZ^+$ is a $b$-elated number if there exists some $m\geq 0$ such that $E_{2,b}^m(a) = 1$.  
An integer $a \in \ZZ^+$ is an \emph{elated number} if there exists some $m \geq 0$ such that $E_{2,10}^m(a) = 1$.

For example, $21$ is an elated number since 
\[E_{2,10}^2(21)= E_{2,10}(2(2^2+1^2)) = E_{2,10}(10) = 1(1^2+0^2) = 1,\] while $46$ is not an elated number since 
\[E_{2,10}^3(46)=E_{2,10}^2(208)=E_{2,10}(136)=46.\] 

There are many similarities and significant differences between $b$-happy numbers and $b$-elated numbers.  For example, all positive integers are both 2-happy~\cite[Theorem 4]{GT01} and 2-elated (Lemma~\ref{L:E2bounded} and Table~\ref{cycletable}, below).

One striking difference is the following frequently used property of the $b$-happy function $S_{2,b}$~\cite{elsedy00,GH22}.  
\begin{lemma}~\label{Slinear}
Fix 
$b \geq 2$.  Let $x$, $y$, and $s\in \ZZ^+$ with $y < b^s$.  Then $S_{2,b}(xb^s + y) = S_{2,b}(x) + S_{2,b}(y)$. 
\end{lemma}
Functions with this property are sometimes called {\em strongly $b$-additive.}  (See, for example~\cite{amri15}.)
While this property does not hold for $E_{2,b}$, the following useful variation does hold.

\begin{lemma}~\label{linear}
Fix $b \geq 2$.  Let $x$, $y$, and $s\in \ZZ^+$ with $x_n$ the leading digit of $x$, and $y < b^s$.
Then \[E_{2,b}(xb^s + y) = E_{2,b}(x) + x_nS_{2,b}(y).\]
\end{lemma}

\begin{proof}
Let 
$x = \sum_{i=0}^n x_i b^i$ and $y = \sum_{i=0}^m y_i b^i$
be the base $b$ expansions of $x$ and $y$, respectively.  Then $m < s$ and 
\begin{align*}
E_{2,b}\!\left(xb^s + y\right)  & =
E_{2,b}\!\left(\sum_{i=0}^n x_i b^{s+i} + \sum_{i=0}^m y_i b^i\right) 
=
x_n \left(\sum_{i=0}^n x_i^2 + \sum_{i=0}^m y_i^2\right) \\
 & = x_n \left(\sum_{i=0}^n x_i^2\right) + x_n\left(\sum_{i=0}^m y_i^2\right) = E_{2,b}(x) + x_nS_{2,b}(y),
\end{align*}
as desired.
\end{proof}

Another significant difference between $S_{2,b}$ and $E_{2,b}$ is their behavior when the base $b$ is odd.  In particular, for $b\geq 3$ odd and $a\in \ZZ^+$, 
$S_{2,b}(a) \equiv a \pmod 2$, while for $E_{2,b}$ the situation is more complicated.

\begin{lemma}\label{oddbase}
Let $b\geq 3$ be odd and let $a\in \ZZ^+$ with $a_n$ the leading digit of $a$.
\begin{enumerate}
\item[(a)]\label{it:oddbase1} If $a_n$ is odd, then $E_{2,b}(a) \equiv a \pmod 2$.
\item[(b)]\label{it:oddbase2} If $a_n$ is even, then $E_{2,b}(a) \equiv 0 \pmod 2$.
\end{enumerate}
In particular, each $b$-elated number is odd and has an odd leading digit.
\end{lemma}

\begin{proof}
Let $a = \sum_{i=0}^n a_i b^i$ be the base $b$ expansion of $a$.
Recalling that $b$ is odd, we have
\begin{align*}
E_{2,b}\left(\sum_{i=0}^n a_i b^i\right)  &=
a_n\left(\sum_{i=0}^n a_i^2\right)\\
 &\equiv a_n\left(\sum_{i=0}^n a_i\right) \equiv
a_n\left(\sum_{i=0}^n a_ib^i\right) \equiv 
a_n a \pmod 2.
\end{align*}
Hence $E_{2,b}(a) \equiv 
a_n a \pmod 2$ and the results follows.
\end{proof}

In the next section, we examine the fixed points and cycles that result from iterations of the $b$-elated function across various bases.  In Section~\ref{sequences}, we introduce results about the possible lengths of arithmetic sequences of $b$-elated numbers and of arithmetic sequences of numbers that under iteration of $E_{2,b}$ map to other specific cycles of $E_{2,b}$.  In Section~\ref{heights}, we investigate the heights of $b$-elated numbers, the height being the minimum number of iterations required to map a $b$-elated number to 1.  Next, in Section~\ref{heights10}, we narrow the study of heights to base 10, determining the smallest elated number of heights 0 through 16. Finally, in Section~\ref{openproblems}, we provide some open questions.

\section{Cycles and Fixed Points}\label{cycles}

In this section, we prove that for each $b \geq 2$ there are only a finite number of cycles resulting from iteration of $E_{2,b}$.  This is achieved by determining a bound above which $E_{2,b}$ is strictly decreasing.  For $2\leq b \leq 10$, we use this bound to find all cycles and fixed points determined by $E_{2,b}$ under iteration.  (See Table~\ref{cycletable}.)

\begin{lemma}\label{L:E2bounded}
    Let $b\geq2$. If $a \geq b^3$, then $E_{2,b}(a)<a$. 
\end{lemma}

The proof parallels that for the $b$-happy function found in~\cite[Lemma 6]{GT01}.

\begin{proof} 
Fix $a \geq b^3$ and let $a =\sum_{i=0}^n a_i b^i$ be the base $b$ expansion of $a$.  
We have
 \begin{align*}
a- E_{2,b}(a) &= \sum_{i=0}^n a_i b^i - a_n \sum_{i=0}^n a_i^2 = \sum_{i=0}^n a_i(b^i - a_na_i) \\
&= 
a_0(1-a_na_0) + a_1(b-a_na_1) \\
&\phantom{=}+ \sum_{i=2}^{n-1} a_i(b^i-a_n a_i) + a_n(b^n -a_n^2)\\
&\geq  a_0(1-a_na_0) + a_1(b-a_na_1) + a_n(b^n -a_n^2),
\end{align*}
since $b^i > a_na_i$ for each $2\leq i <n$.  Expanding and noting that $a_0$ and $a_1b$ are nonnegative and $n \geq 3$ yields
 \[a- E_{2,b}(a) \geq  a_nb^n -a_na_0^2 -a_na_1^2 -a_n^3
 = a_n(b^n-a_0^2 -a_1^2 -a_n^2) > 0. \qedhere
 \]
\end{proof}

\begin{table}[b!]
\begin{center}
\begin{tabular}{|c|l|}\hline
Base $b$ & Fixed Points and Cycles (written in base $b$)
\\
\hline \hline
2 & (1) \\ \hline
3 & (1), (12), (20, 22, 121)  \\ \hline
4 & (1), (20) \\ \hline
5 & (1), (13, 20) \\
 \hline
6 & (1), (50, 325, 310), (53, 442, 400, 144) \\
 \hline
7 & (1), (13), (22), (505), (2, 11)\\
\hline
8 & (1), (536), (660), (36, 207, 152), (5, 175, 113, 13, 12) \\
\hline
9 & (1), (30), (646), (762) \\ \hline
10 & (1), (298), (46, 208, 136),
(26, 80, 512, 150), \\ &
(33, 54, 205, 58, 445, 228, 144)\\ \hline
\hline
\end{tabular}
\caption{Fixed points and cycles of $E_{2,b}$ for $2 \leq b \leq 10$}
\label{cycletable}
\end{center}
\end{table}

Lemma~\ref{L:E2bounded} implies that under iteration of $E_{2,b}$,  every positive integer enters one of finitely many cycles of finite length, each containing at least one number less than $b^3$.  Hence, a direct computer calculation yields the following theorem.

\begin{theorem}
The fixed points and cycles that result from iterating $E_{2,b}$ for $2\leq b \leq 10$ are as given in Table~\ref{cycletable}.
\end{theorem}

In the same way, the cycles corresponding to larger bases can be found, limited only by the capacity of the computer and software.

\section{Sequences of Consecutive Elated Numbers}\label{sequences}

In this section, we investigate the existence of arithmetic sequences of $b$-elated numbers and of some closely related classes of numbers.

Grundman and Teeple~\cite{GT07} defined a {\em $d$-consecutive} sequence to be an arithmetic sequence 
with constant difference $d$.  For a fixed base
$b\geq 2$, setting $d = \gcd(2,b-1)$, they proved that there exist arbitrarily long finite $d$-consecutive sequences of $b$-happy numbers.  Since, when $b$ is odd, only odd numbers can be $b$-happy, this result is the best possible.

Given a fixed base $b\geq 2$, $U_{2,b}$ is the set of all numbers in cycles of the $b$-happy function $S_{2,b}$ and for any $u\in U_{2,b}$, $a\in \ZZ^+$ is a {\em $u$-attracted number for $S_{2,b}$} if for some $k\in \ZZ^+$, $S_{2,b}^k(a) = u$.  As noted in~\cite{GH22}, the main proofs in~\cite{GT07} trivially extend to $d$-consecutive sequences of $u$-attracted numbers for $S_{2,b}$. 

In a similar vein, for $b\geq 2$, we define 
$U_{2,b}^E$ to be the set of all numbers in cycles of $E_{2,b}$, that is,
\[ U_{2,b}^E = \left\{ a \in \ZZ^+ \mid \text{ for some } m\in \ZZ^+, \,  E_{2,b}^m(a) = a \right\}.
\]
For $u\in U_{2,b}^E$, we say that $a\in \ZZ^+$ is a {\em $u$-attracted number for $E_{2,b}$} if for some $k\in \ZZ^+$, $E_{2,b}^k(a) = u$.  
So a $b$-elated number is the same as a $1$-attracted number for $E_{2,b}$.

A concept in~\cite{GT07} that is key for our main results in this section is that of a $(2,b)$-good set.  Specifically, for $b\geq 2$, a finite set $T$ is {\em $(2,b)$-good}
if, for each $u \in U_{2,b}$, there exist
$n$, $k \in {\ZZ}^+$ such that for each $t \in T$, $S_{2,b}^k(t+n) = u$. 
\begin{theorem}[Grundman, Teeple]
\label{goodthm}
Fix $b \geq 2$ and let $d = \gcd(2,b-1)$.
A finite set $T$ of positive integers is $(2,b)$-good if and only if 
all of the elements of $T$ are congruent modulo $d$.
\end{theorem}

For $b\geq2$ and $r \geq 0$, define $R_{r,b}: \ZZ_{\geq  0} \rightarrow \ZZ_{\geq  0}$ by
\[R_{r,b}(x) = \sum_{i=r}^{r+x-1} b^i.
\]
Note that for each $r\geq 0$ and $x\in \ZZ^+$, $E_{2,b}(R_{r,b}(x)) = S_{2,b}(R_{r,b}(x)) = x$.

Utilizing Theorem~\ref{goodthm} and the function $R_{r,b}(x)$, 
we prove a property of $(2,b)$-good sets and the function $E_{2,b}$. 

\begin{lemma}\label{Egood}
Let $b\geq 2$ and let $T$ be a $(2,b)$-good set.  Then for each $u\in U_{2,b}^E$, there exist $m$, $k \in {\ZZ}^+$ such that for each $t \in T$, 
\[
E_{2,b}^{k}(t+m) = u.
\] 

Specifically, let
$n$, $k \in {\ZZ}^+$ such that for each $t \in T$, $S_{2,b}^{k-1}(t+n) = 1$, and for some $r\in \ZZ^+$ such that 
$b^r > \max\left\{S_{2,b}^i(t+n)\mid 0\leq i \leq k-1, t\in T\right\}$, set $m = R_{r,b}^{k-1}(R_{r,b}(u)-1) + n$.  
\end{lemma}

\begin{proof}
Let $b$, $T$, $u$,  $n$, $k$, $r$, and $m$ be as in the statement of the lemma. 

We use induction on $j$ to show that for each $1 \leq j \leq k$, 
\[E_{2,b}^{j}\!\left(R_{r,b}^{j-1}(R_{r,b}(u)-1) + S_{2,b}^{k-j}(t+n)\right) = u.\]  
First, for every $t \in T$,
\[E_{2,b}^1\!\left(R_{r,b}^0(R_{r,b}(u)-1)+S_{2,b}^{k-1}(t+n)\right)=E_{2,b}\!\left(R_{r,b}(u)-1+1\right) = u.\]

Next, assume that for some $2 \leq j\leq k$, 
\[
E_{2,b}^{j-1}\!\left(R_{r,b}^{j-2}(R_{r,b}(u)-1) + S_{2,b}^{k-j+1}(t+n)\right)=u.
\]
Then by Lemma~\ref{linear} and the choice of $r$,
\begin{align*}
    &\phantom{=}E_{2,b}^{j}\!\left(R_{r,b}^{j-1}(R_{r,b}(u)-1) + S_{2,b}^{k-j}(t+n)\right)\\
   &= E_{2,b}^{j-1}\!\left(E_{2,b}\!\left(R_{r,b}^{j-1}(R_{r,b}(u)-1)\right) + 1 \cdot S_{2,b}\!\left(S_{2,b}^{k-j}(t+n)\right)\right)\\
    &= E_{2,b}^{j-1}\!\left(R_{r,b}^{j-2}(R_{r,b}(u)-1) + S_{2,b}^{k-j+1}(t+n)\right)\\
    &=u.
\end{align*}
Thus for all $ t \in T$,
\[E_{2,b}^k(t+m) = E_{2,b}^k\!\left(R_{r,b}^{k-1}(R_{r,b}(u)-1) + S_{2,b}^0(t+n)\right) =u.\qedhere\]
\end{proof}

The following theorem follows from Theorem~\ref{goodthm} and Lemma~\ref{Egood}.

\begin{theorem}\label{T:d-u-attracted}
Let $b \geq 2$ and let $u\in U_{2,b}^E$.  Setting $d = \gcd(2, b-1)$, there exist arbitrarily long finite $d$-consecutive sequences of $u$-attracted numbers for $E_{2,b}$.
\end{theorem}

\begin{proof}
Fix $b \geq 2$ and $u\in U_{2,b}^E$.  Let $L\in \ZZ^+$ be the desired sequence length.  Set $T = \{d, 2d, \dots, Ld\}$. By Theorem~\ref{goodthm}, $T$ is $(2,b)$-good. So, by Lemma~\ref{Egood}, there exist $m$, $k \in \ZZ^+$ such that for each $t \in T$, $E_{2,b}^k(t+m)=u$. Therefore, $\{d+m, 2d+m, \dots, Ld+m\}$ is a $d$-consecutive sequence of $u$-attracted numbers for $E_{2,b}$.
\end{proof}

Setting $u = 1$ in Theorem~\ref{T:d-u-attracted} yields that there exist arbitrarily long finite $d$-consecutive sequences of $b$-elated numbers.  As with $b$-happy numbers, this result is the best possible since, by Lemma~\ref{oddbase}, for $b$ odd all $b$-elated numbers are odd.

In the more general setting of $u$-attracted numbers for $E_{2,b}$, however, Theorem~\ref{T:d-u-attracted} is not the best possible since $E_{2,b}$ can map odd numbers to even numbers.  As seen in the following theorem, for even $u\in U_{2,b}^E$ consecutive sequences of $u$-attracted numbers for $E_{2,b}$ always exist, regardless of the parity of $b$.

\begin{theorem}\label{T:evenconsec}
Given $b \geq 2$ and any even $u\in U_{2,b}^E$, there exist arbitrarily long finite sequences of consecutive $u$-attracted numbers for $E_{2,b}$.
\end{theorem} 

\begin{proof}
    For $b$ even, the result is immediate from Theorem~\ref{T:d-u-attracted}.  So let
$b\geq 3$ be odd. Fix $L \in \ZZ^+$, the desired sequence length, and let $u\in U_{2,b}^E\cap2\ZZ^+$. Set 
\[
M=2\cdot\max_{1\leq i\leq L}\{S_{2,b}(i)\}.
\]
Set $T = \{2j+8\mid 0\leq j\leq M\}$.  By Theorem~\ref{goodthm}, $T$ is $(2,b)$-good and so we can fix $n$, $k$, $r$, and $m$ as in the statement of Lemma~\ref{Egood}.

Since $b$ is odd, $S_{2,b}$ and $R_{r,b}$ preserve parity. Letting $t \in T$, we have $t$ is even and $S_{2,b}^{k-1}(t+n)=1$, implying that $n$ is odd. Further, since $u$ is even, $R_{r,b}^{k-1}(R_{r,b}(u)-1)$ is odd. Hence $m = R_{r,b}^{k-1}(R_{r,b}(u)-1) + n$ is even.
Set $m^\prime = m/2$.

Fix $w\in \ZZ^+$ such that $b^w>L$, and define
\[A = 2b^{m^\prime+w} + R_{w,b}(m^\prime).\]
Now, let $h \in \ZZ^+$ such that $h \leq L$. Then,
using Lemmas~\ref{Slinear} and~\ref{linear},
\begin{align*}
E_{2,b}(A+h) & = E_{2,b}(2b^{m^\prime+w} + R_{w,b}(m^\prime) + h) \\
& = E_{2,b}(2) + 2\left(S_{2,b}(R_{w,b}(m^\prime)) + S_{2,b}(h)\right) \\
& = 8 + 2\left(m^\prime+S_{2,b}(h)\right)\\
& = 8+m+2S_{2,b}(h),
\end{align*}
which is even.
Since $0 < 2S_{2,b}(h)\leq M$, we have $8+m< E_{2,b}(A+h)\leq 8+m+M < 8 + m + 2M$.  Therefore, for some $t\in T$, $E_{2,b}(A+h) = t + m$.  So, by Lemma~\ref{Egood}, $E^{k+1}_{2,b}(A+h) = E^{k}_{2,b}(t+m) = u$.  Hence, 
$A+1, A+2, \dots, A+L$ is a sequence of $L$ consecutive $u$-attracted numbers for $E_{2,b}$, as required.
\end{proof}

We conclude with a result on consecutive non-$b$-elated numbers.

\begin{corollary}
Let $b\geq 2$.  If there exists some $u \in U_{2,b}^E - \{1\}$, then 
there exist arbitrarily long finite sequences of consecutive numbers that are not $b$-elated.
\end{corollary}

\begin{proof}
If $b$ is even, the result follows directly from Theorem~\ref{T:d-u-attracted}.  If $b$ is odd and $u$ is even, it follows directly from Theorem~\ref{T:evenconsec}.  If both $b$ and $u$ are odd, then Theorem~\ref{T:d-u-attracted} yields the existence of a $2$-consecutive sequence of arbitrary finite length $L \in \ZZ^+$ of $u$-attracted numbers for $E_{2,b}$, say $a, a+2, \dots, a+(2L-2)$.  By Lemma~\ref{oddbase}, $a$ is odd.  Also by Lemma~\ref{oddbase}, the even numbers $a+1,a+3,\dots, a+(2L-1)$ must be mapped by $E_{2,b}$ to even numbers.  Hence, none of the numbers in the sequence  $a, a+1, \dots, a+(2L-1)$ of length $2L$ are $b$-elated numbers.
\end{proof}

\section{Heights of \texorpdfstring{$b$}{b}-Elated Numbers}\label{heights}

In this section, we study the number of iterations of $E_{2,b}$ necessary for a given $b$-elated number to reach 1.  The {\it height} of a $b$-elated number $a$ is the smallest number $m \geq 0$ such that $E_{2,b}^m(a) = 1$.  We first consider the general case in which the base $b \geq 2$ is arbitrary, then in Section~\ref{heights10}, we narrow our attention to the base $10$ case.  Of particular interest is the smallest positive integer of each height. 
(For analogous results for $b$-happy numbers, see~\cite{cai08,GH22,GH22C,GT03,mei18,OEIS}.)

For each $k\geq 0$ and $b\geq 2$, let
$\varepsilon_{k,b}$ denote the smallest $b$-elated number of height $k$.  Clearly, for each $b\geq 2$, $\varepsilon_{0,b} = 1$ and
$\varepsilon_{1,b} = b$. The values of $\varepsilon_{k,b}$ for small values of $k$ and $2\leq b \leq 10$, found by a direct computer search, are given in Table~\ref{smallbaseheights}.

\begin{table}[hbt!]
\begin{center}
\begin{tabular}{|c|l|} \hline
    $b$ & $\varepsilon_{k,b}$ (written in base $b$) for $k=2, 3,\dots$ \\ \hline
2 &
11, 111, 1111111 \\ \hline
3 &
111, 1222, 12222222222222\\
\hline 
4 &
22, 13, 122, 23, 113, 3, 111, 333, 3222, 31123333333 \\
\hline
5 &
12, 3, 34, 133, 3444, 3334444444444444444\\
\hline
6 &
112, 233, 3233, 4555555555\\
\hline
7 &
1112, 1266666666666\\
\hline
8 &
2, 11, 47, 32, 15, 75, 55, 277, 57, 146, 35, 367, 2577, 76677 \\
\hline
9 &
122, 5, 12, 113, 1666, 527788\\
\hline
10 &
13, 51, 67, 97, 668, 77, 746, 92, 717, 5369, 8888999999 \\
\hline
\end{tabular}
\caption{Smallest $b$-elated numbers of small heights for $2 \leq b \leq 10$} 
\label{smallbaseheights}
\end{center}
\end{table}

It is easy to see that for $k\geq 2$ 
the non-leading base $b$ digits of $\varepsilon_{k,b}$ are nonzero and in nondecreasing order.  With this in mind, we define a \emph{$b$-basic} number to be any number greater than $b$ with non-leading base $b$ digits nonzero and in nondecreasing order. If all of the base $b$ digits of a number are in nondecreasing order, including the leading digit, we call that number \emph{fully $b$-basic}.

Noting that a $2$-basic number is, by definition, a string of ones, the following is immediate.
\begin{theorem}\label{base2}
For $k\geq 1$ 
\[\varepsilon_{k+1,2} = \sum_{i=0}^{\varepsilon_{k,2} - 1} 2^i=2^{\varepsilon_{k,2}}-1.\]
\end{theorem}

Similarly, we have a concrete characterization of $\varepsilon_{k,3}$, which parallels a result on $3$-happy numbers of minimal height found in~\cite[Lemma 5.1]{lap}.
\begin{theorem}\label{T:base3elated}
    For $k\geq2$
    \[
    \varepsilon_{k+1,3} 
    = 2\cdot3^{\frac{\varepsilon_{k,3}-1}{4}}-1.
    \]
\end{theorem}

\begin{proof}
By Table~\ref{smallbaseheights}, $\varepsilon_{2,3}=111$ ($13$ in base $10$) and $\varepsilon_{3,3}=1222$ ($53$ in base $10$).  Since $53=2\cdot3^{\frac{13-1}{4}}-1$, the theorem is true for $k=2$. 

Proceeding by induction, assume that for some $k\geq 3$
    \[
    \varepsilon_{k,3}=2\cdot3^{\frac{\varepsilon_{k-1,3}-1}{4}}-1.
    \]
Note that $\varepsilon_{k,3} \equiv 1 \pmod 4$ and therefore we can define an integer
\[x=2\cdot3^{\frac{\varepsilon_{k,3}-1}{4}}-1
= 1\cdot 3^{\frac{\varepsilon_{k,3}-1}{4}} +
\sum_{i=0}^{(\varepsilon_{k,3}-1)/4-1} 2\cdot 3^i.\] 
Then
    \[
    E_{2,3}(x)=1\cdot\left(1^2+\frac{\varepsilon_{k,3}-1}{4}\cdot 2^2\right)=\varepsilon_{k,3},
    \]
    and therefore $x$ is a $3$-elated number of height $k+1$.

Suppose for a contradiction that $\varepsilon_{k+1,3} < x$.  Then $\varepsilon_{k+1,3}$ has at most as many digits as $x$ has and, by Lemma~\ref{oddbase}, $\varepsilon_{k+1,3}$ has leading digit 1.  Therefore, 
$E_{2,3}(\varepsilon_{k+1,3}) < E_{2,3}(x) = \varepsilon_{k,3}$, making $E_{2,3}(\varepsilon_{k+1,3})$ a height $k$ 3-elated number that is smaller than $\varepsilon_{k,3}$, a contradiction.  Hence $\varepsilon_{k+1,3} = x$, as desired.
\end{proof}
We note that, by similar reasoning, for bases $b=2$ and $b=3$ and each $k\geq 0$, $\varepsilon_{k,b}$ is also the smallest $b$-happy number of height $k$. (See also~\cite[Corollary 4.1, Lemma 5.1]{lap}.)

To help understand the properties and in some cases the actual values of $\varepsilon_{k,b}$, we consider specific sets of preimages under $S_{2,b}$. For $b \geq 3$ and $a\in \ZZ^+$, let $\mathscr S_{2,b}(a)$ denote the set of fully $b$-basic numbers of shortest length in the preimage of $a$ under $S_{2,b}$.  We begin with a theorem about the elements of $\mathscr S_{2,b}(a)$ for each $a\in \ZZ^+$. 

\begin{theorem}\label{T:non9theorem2}
    Let $b\geq 3$, $a\in \ZZ^+$, and $u\in \mathscr S_{2,b}(a)$. The number of non-$(b-1)$ digits in the base $b$ representation of $u$ is strictly less than $2b$.
\end{theorem}

\begin{proof}
Fix $b$, $a$, and $u$ as in the theorem.  
Let $t\geq 0$ denote the number of non-$(b-1)$ digits in the base $b$ representation of $u$. Then $u$ is of the form
\[u=a_1a_2\cdots a_t \underbrace{(b-1)(b-1)\cdots(b-1)}_{\ell}\] with each $a_i< b-1$ and $\ell \geq 0$.

Fix $q\geq 0$ and $0\leq r<(b-1)^2$ such that 
    \begin{equation*}
        a_1^2+a_2^2+\cdots+a_t^2=(b-1)^2q+r.
    \end{equation*} 
By Lagrange's Theorem, there exist nonnegative integers $c_1$, $c_2$, $c_3$, $c_4$, such that $r=c_1^2+c_2^2+c_3^2+c_4^2$.  Since $r<\left(b-1\right)^2$, we necessarily have that each $c_i<b-1$.
Thus
    \begin{align*}
    a & = 
    S_{2,b}(u) \\ 
    &=a_1^2+a_2^2+\cdots+a_t^2+(b-1)^2\ell \\
    & = (b-1)^2q+r+(b-1)^2\ell \\
    & = c_1^2+c_2^2+c_3^2+c_4^2 +(b-1)^2q + (b-1)^2\ell \\
    & = S_{2,b}(m),
    \end{align*}
 where $m$ is the number with base $b$ representation given by
\[m = c_1c_2c_3c_4\underbrace{(b-1)(b-1)\cdots(b-1)}_{q}\underbrace{(b-1)(b-1)\cdots(b-1)}_{\ell}.\]
Now $u$ has $t+\ell$ digits  
while $m$ has at most $4+q+\ell$ digits.  
Since $\mathscr S_{2,b}(a)$ consists of numbers of minimal length mapping to $a$, $t+\ell \leq 4+q+\ell$, implying that $t - 4 \leq q$.

Recalling that each 
$a_i< b-1$, we have 
\[(b-1)^2(t-4) \leq (b-1)^2q \leq a_1^2+a_2^2+\cdots+a_t^2\leq (b-2)^2t,\]
and so, since $2b-3 > 1$,
\[t \leq \frac{4(b-1)^2}{(b-1)^2-(b-2)^2}
=2b-1+\frac{1}{2b-3} < 2b,
\] 
as required.
\end{proof}

The following corollary on $b$-happy numbers of smallest heights is immediate.  The case of $b = 10$ was proved previously by
Cai and Zhou~\cite[Corollary to Theorem 1]{cai08}.  (See also~\cite[Theorem 3.1]{mei18}.)

\begin{corollary}\label{C:numberofnonb-1}
    For any $b\geq 3$ and $k\geq 1$, the number of non-$(b-1)$ digits in 
    the minimal $b$-happy number of height $k$ 
    is strictly less than $2b$.
\end{corollary}

Applying Theorem~\ref{T:non9theorem2} to $b$-elated numbers yields the following.

\begin{corollary}\label{T:non9theorem}
    For any $b\geq 3$ and $k\geq 1$, the number of non-leading non-$(b-1)$ digits in the base $b$ representation of $\varepsilon_{k,b}$ is strictly less than $2b$.
\end{corollary}

\begin{proof}
Let $d$ be the leading digit and let $\varepsilon^\prime$ be the number formed by the non-leading digits of the base $b$ representation of $\varepsilon_{k,b}$. 
Then $\varepsilon^\prime$ is fully $b$-basic and 
$S_{2,b}(\varepsilon^\prime) = E_{2,b}(\varepsilon_{k,b})/d - d^2$.

Let $u$ be an arbitrary fully $b$-basic positive integer such that $S_{2,b}(u) = S_{2,b}(\varepsilon^\prime)$.  Then letting $v$ be the number formed by prepending $d$ onto $u$, we have 
\[E_{2,b}(v) = d(d^2 + S_{2,b}(u)) = d(d^2 + S_{2,b}(\varepsilon^\prime)) = E_{2,b}(\varepsilon_{k,b}),
\]
implying that $v$ is a $b$-elated number of height $k$. Thus $\varepsilon_{k,b} \leq v$, and so $\varepsilon' \leq u$.  Since $u$ was arbitrary, it follows that $\varepsilon' \in \mathscr S_{2,b}(E_{2,b}(\varepsilon_{k,b})/d - d^2)$.  By Theorem~\ref{T:non9theorem2} the number of non-$(b-1)$ digits in the base $b$ representation of $\varepsilon'$ is strictly less than $2b$, and the corollary follows.
\end{proof}

With Theorem~\ref{General S theorem}, below, we reduce
the determination of the sets $\mathscr S_{2,b}(a)$ to only those for  small values of $a$ and a short calculation.   
In Section~\ref{heights10}, we demonstrate this, using the values of $\mathscr S_{2,10}(a)$ for small $a$ in computing values of $\varepsilon_{k,10}$ for $k \leq 16$.

For each $b\geq 3$, define the set
\[T_b = \{S_{2,b}(n) \vert 1\leq n < b^{2b}\},\]
and note that for each $ m \in \ZZ^+ - T_b$, every element of
$\mathscr S_{2,b}(m)$ has more than $2b$ digits and hence by Theorem~\ref{T:non9theorem2}, ends in the digit $b-1$.  Fix 
$a^*_b \in T_b$ to be the largest positive integer $a$ such that some element of 
$\mathscr S_{2,b}(a)$ does not end in the digit $b-1$.
Set \[C_b = \left\lfloor\frac{a^*_b}{\left(b-1\right)^2}\right\rfloor.\] 
 Observe that $C_b(b-1)^2 \leq a^*_b < (C_b+1)(b-1)^2$.

We provide the values of $a_b^*$ and $C_b$ for small values of $b$ in Table~\ref{Cb}.

\begin{table}[hbt!]
\begin{center}
\begin{tabular}{|c|r|r|} \hline
    $b$ & $a_b^*$ & $C_b$  \\ \hline
3 & 3 & 0\\\hline
4 & 16 & 1\\\hline
5 & 31 & 1\\\hline
6 & 128 & 5\\\hline
7 & 191 & 5\\\hline
8 & 324 & 6\\\hline
9 & 368 & 5\\\hline
10 & 561 & 6\\\hline
\end{tabular}
\caption{Values of $a^*_b$ and $C_b$ for $3\leq b \leq 10$} 
\label{Cb}
\end{center}
\end{table}

\begin{theorem}\label{General S theorem}
Fix $b\geq3$ and 
$a > C_b\left(b-1\right)^2$. 
Set
 \[
q = \left\lfloor\frac{a-1}{\left(b-1\right)^2}\right\rfloor  - C_b \mbox{\ \ \ and\ \ \ }
a^\prime = a - \left(b-1\right)^2q.
\]
Then, the elements of $\mathscr S_{2,b}(a)$ are precisely the elements of $\mathscr S_{2,b}(a^\prime)$ with $q$ $(b-1)$'s appended to each. That is,
\[\mathscr S_{2,b}(a) = \left\{w b^q + (b-1)R_{0,b}(q) \vert w\in \mathscr S_{2,b}(a^\prime) \right\}.
\]
\end{theorem}

\begin{proof}
Let $b\geq3$ and 
$a > C_b\left(b-1\right)^2$.  Suppose for a contradiction that $a$ is minimal such that the theorem does not hold. 
If $q = 0$, then $a = a^\prime$ and the theorem is trivially true.  So $q \geq 1$ and thus 
$a > (C_b+1)\left(b-1\right)^2>a^*_b$.  

Let $\tilde{a} = a - (b-1)^2$ and note that 
$C_b\left(b-1\right)^2 < \tilde a < a$.  
By the minimality of $a$, the theorem holds for $\tilde a$.  Hence, 
letting $q_{\tilde{a}}$ and $\tilde{a}^\prime$ represent the values of $q$ and $a^\prime$ corresponding to $\tilde a$, we have
\[
q_{\tilde{a}} = \left\lfloor\frac{a - (b-1)^2-1}{\left(b-1\right)^2}\right\rfloor  - C_b = q-1 \geq 0, 
\]
\[
\tilde{a}^\prime = \tilde{a} - (b-1)^2q_{\tilde{a}} = a - (b-1)^2 - \left(b-1\right)^2(q-1)
= a^\prime,
\]
and therefore $\mathscr S_{2,b}(\tilde{a}) = \left\{w b^{q-1} + (b-1)R_{0,b}(q-1) \vert w\in \mathscr S_{2,b}(a^\prime) \right\}$.
Note further that, since $a > a^*_b$,  each element of $\mathscr S_{2,b}(a)$ ends in the digit $b-1$, and since $a > (b-1)^2$, each element of $\mathscr S_{2,b}(a)$ is strictly greater than $b-1$.

Now, given $c\in \mathscr S_{2,b}(a)$, let $c^-$ be the number obtained by dropping the last digit of $c$, \[c^- = (c- (b-1))/b > 0.\]   Since $S_{2,b}(c^-) = a - (b-1)^2 = \tilde a$, if $c^- \notin \mathscr S_{2,b}(\tilde{a})$, then each $d \in \mathscr S_{2,b}(\tilde{a})$ has fewer digits than $c^-$.  Letting $d^+$ be the number obtained by appending the digit $b-1$ to the end of $d$, $d^+ = db + (b-1)$, we have 
$S_{2,b}(d^+) = \tilde a + (b-1)^2 = a$.  But since $d$ has fewer digits than $c^-$, $d^+$ has fewer digits than $c$.  This is a contradiction, since $c\in \mathscr S_{2,b}(a)$.  Hence $c^- \in \mathscr S_{2,b}(\tilde{a})$.  

Conversely, given $d \in \mathscr S_{2,b}(\tilde a)$, if $d^+ \notin \mathscr S_{2,b}(a)$, then each $c \in \mathscr S_{2,b}(a)$ has fewer digits than $d^+$.  But then $c^-$ has fewer digits than $d$, a contradiction.  Hence $d^+ \in \mathscr S_{2,b}(a)$.  

Thus $c\in \mathscr S_{2,b}(a)$ if and only if $c^- \in \mathscr S_{2,b}(\tilde a)$
if and only if for some $w\in \mathscr S_{2,b}(a^\prime)$, $c^- = w b^{q-1} + (b-1)R_{0,b}(q-1)$ if and only if for some $w\in \mathscr S_{2,b}(a^\prime)$, $c = w b^q + (b-1)R_{0,b}(q)$.  Therefore, 
\[\mathscr S_{2,b}(a) = \left\{w b^q + (b-1)R_{0,b}(q) \vert w\in \mathscr S_{2,b}(a^\prime) \right\},
\]
as desired.
\end{proof}

As an example of Theorem~\ref{General S theorem}, consider $b=10$ and $a = 731$.  Then $q = \left\lfloor\frac{731-1}{9^2}\right\rfloor - 6 = 3$ and $a^\prime = 731 - 81\cdot 3 = 488$.  By Table~\ref{S(a) table},
$\mathscr S_{2,10}(488) = \{6889999\}$ and so, by the theorem, $\mathscr S_{2,10}(731) = 
\left\{w \cdot 10^3 + 9\cdot R_{0,10}(3) \vert w\in \mathscr S_{2,10}(488)\right\} = 
\{6889999999\}$.

More generally, Theorem~\ref{General S theorem} shows that other than the appended $(b-1)$'s, the sets $\mathscr{S}_{2,b}$ enter a cycle of length $(b-1)^2$. This observation leads to the following corollary.

\begin{corollary}\label{Period S}
Given $a\geq 1$, let $\mathscr{P}_{2,b}(a)$ denote the 
set of strings of digits (possibly including the empty string) obtained from $\mathscr{S}_{2,b}(a)$ by removing the trailing $(b-1)$'s from each element. 
The sequence of sets $\mathscr{P}_{2,b}(a)$  
indexed by $a> C_b(b-1)^2$ is periodic with period $(b-1)^2$.
\end{corollary}

\begin{proof}
Let $a> C_b(b-1)^2$. 
Using the notation of Theorem~\ref{General S theorem},
the elements of $\mathscr{S}_{2,b}(a)$ are precisely the elements of $\mathscr{S}_{2,b}(a')$ with $q$ $(b-1)$'s appended to each. 
Therefore, $\mathscr{P}_{2,b}(a)=\mathscr{P}_{2,b}(a')$. Since the value of $a'$ depends only on the congruence class of $a$ modulo $(b-1)^2$, 
the sequence of sets is periodic with period $(b-1)^2$, as desired.
\end{proof}

\section{Heights in Base \texorpdfstring{$10$}{10}}\label{heights10}

In this section, we restrict our focus to base 10, determining the values of $\varepsilon_{k,10}$ for each $k\leq 16$.  To simplify terminology and notation, we generally suppress the prefix or subscript ``10."  
(For example, we refer to a 10-basic number as being ``basic.")
Define two numbers to be \emph{equivalent} if their base 10 representations have the same leading digit and the digits of one is a permutation of the digits of the other.  Note in particular that two equivalent numbers have the same image under $E_2 = E_{2,10}$.

Cai and Zhou~\cite{cai08} proved that for $k \geq 7$, $S_2 = S_{2,10}$ maps $\sigma_{k}$, the least happy number of height $k$, to $\sigma_{k-1}$.  They used this result and bounds on the number of digits in $\sigma_k$ to develop a method for determining $\sigma_{k}$ from $\sigma_{k-1}$.  
 
For elated numbers, on the other hand, it appears that there does not exist a value such that for each larger $k$, $E_2$ maps $\varepsilon_{k}$ to $\varepsilon_{k-1}$.  
Instead, for a given height, we identify a small set of numbers into which the smallest number of the next larger height must map.  
We use this approach in each of the following proofs, the proofs growing more complex as the numbers involved get significantly larger. 

The values of  $\varepsilon_k$ for $0 \leq k \leq 12$ are given in Table~\ref{10heighttable}.  These were determined via a direct computer search, using the fact that for $k \geq 2$, $\varepsilon_{k}$ must be basic.  
\begin{table}[htb]
\begin{center}
\begin{tabular}{|c|c|} \hline
    Height $k$ & $\varepsilon_{k}$ \\ \hline
    0 & 1 \\  \hline
    1 & 10 \\ \hline
    2 & 13 \\ \hline
    3 & 51 \\ \hline
    4 & 67 \\ \hline
    5 & 97 \\ \hline
    6 & 668 \\ \hline
    7 & 77 \\ \hline
    8 & 746 \\ \hline
    9 & 92 \\ \hline
    10 & 717 \\ \hline
    11 & 5369 \\ \hline
    12 & 8888999999\\ \hline 
\end{tabular}
\caption{Smallest elated numbers of height $k$ for $0\leq k \leq 12$} 
\label{10heighttable}
\end{center}
\end{table}

To compute the values of $\varepsilon_k$ for higher values of $k$, we use the following corollary to Theorem~\ref{General S theorem}, noting that from
Table~\ref{Cb}, $C_{10}=6$.

\begin{corollary}\label{S theorem}
Fix $a > 486$.  Set
\[q = \left\lfloor\frac{a-1}{81}\right\rfloor  - 6 {\ \ \ and\ \ \ }
a^\prime = a - 81q.
\]
Then the elements of $\mathscr S_{2}(a)$ are precisely the elements of $\mathscr S_{2}(a^\prime)$ with $q$ 9's appended to each:
\[\mathscr S_{2}(a) = \left\{w \cdot 10^q + 9\cdot R_{0,10}(q) \vert w\in \mathscr S_{2}(a^\prime) \right\}.
\]
\end{corollary}

It follows that the sets $\mathscr S_{2}(a)$ for large values of $a$ are determined by those for $487 \leq a \leq 567$, which are easily found with a direct computer search.  We record some of these in Table~\ref{S(a) table}, including the values that are used in the following proofs.   

\begin{table}[hbt!]
\begin{center}
\begin{tabular}{|c|c|} \hline
$a$ & $\mathscr S_{2}(a)$ \\\hline
487 & 
\{1999999\}\\
\hline
488 & 
\{6889999\}\\
\hline
529 & 
\{88888889\}\\
\hline
534 & 
\{18899999,
47899999\}\\
\hline
543 & 
\{57899999\}\\
\hline
546 & 
\{88888899\}\\
\hline
549 & 
\{48899999\}\\
\hline
557 & 
\{378888999,
458889999,
466899999\}\\
\hline
561 & 
\{157999999,
368889999,
377799999,
555999999,
788888888\} \\
\hline
564 & 
\{188888999,
257999999,
478888999,
567799999\}\\
\hline
567 & 
\{9999999\}\\
\hline
\end{tabular}
\caption{Examples of $\mathscr S_{2}(a)$} 
\label{S(a) table}
\end{center}
\end{table}

Although we do not use these results, we note that it follows from Corollary~\ref{S theorem} and the computation of the values of $\mathscr S_{2}(a)$ for $486< a\leq 567$ that, for $k\geq 1$, the number of non-leading non-9 digits of $\varepsilon_{k}$ is strictly less than 8.  Further, if the number of these digits is 6 or 7, the digits form one of the strings $378888$, $188888$, $888888$, or $8888888$.

The remainder of this section is focused on providing, with proof, the values
of $\varepsilon_k$ for $13\leq k\leq 16$.  
We assume the use of a standard
computer algebra package for basic computations with relatively
small numbers, and explain the details of the calculations with larger numbers. 

A straightforward computer search yields the following lemma, which is used in proving 
Theorem~\ref{T:height13}.

\begin{lemma}\label{L:ht12-digits}
The only basic elated number of height $12$ less than $8\times 10^{10}$ is $\varepsilon_{12}=8888999999$.
\end{lemma}

\begin{theorem}\label{T:height13}
The smallest elated number of height $13$ is \[\varepsilon_{13} = 8158 \times 10^{13888887}-1,\]
and the other basic elated numbers of height $13$ less than $8\times 10^{13888891}$ are 
\begin{align*}
     8368890000\times 10^{13888881}-1, \hspace{1.2cm} & 8377800000\times 10^{13888881}-1, \\  8556000000\times 10^{13888881}-1, 
     \mbox {\ \,   and\ \ \ } & 8788888889\times 10^{13888881}-1.
\end{align*}
\end{theorem}

\begin{proof}
A direct computation shows that each of these is an elated number of height 13.  To see that the list is complete, let $x < 8\times 10^{13888891}$ be a basic elated number of height 13.  Suppose the leading digit of $x$ is less than $8$.  Then
\[
E_2(x)\leq 7(7^2+13888891\cdot 9^2) 
< \varepsilon_{12},
\]
a contradiction since $E_2(x)$ is of height 12 and 
$\varepsilon_{12}$ is minimal of height 12.  Hence the leading digit of $x$ is 8 or 9 and therefore $x$ has at most $13888891$ digits.
Thus,
\[
E_2(x) \leq  9(9^2+ 13888890\cdot 9^2) 
<8\times 10^{10},
\] 
and so by Lemma~\ref{L:ht12-digits},
$E_2(x)$ is equivalent to $\varepsilon_{12} = 8888999999$. Since the leading digit of $x$ divides $E_{2}(x)$ and none of the numbers equivalent to $8888999999$ is a multiple of 9,
the leading digit of $x$ is 8.
Therefore
$E_2(x) = 8999999888$, the only multiple of 8 equivalent to $\varepsilon_{12}$.  
It follows that the number formed by removing the leading digit of $x$ is mapped by $S_{2}$ to $E_2(x)/8 - 8^2 = 1124999922$.

Applying Corollary~\ref{S theorem} to $a=1124999922$, then using
Table~\ref{S(a) table}, we have 
$q = 13888881$ and
\begin{align*}
\mathscr{S}_{2}(a^\prime) &= \mathscr{S}_{2}(561) \\&=\{157999999,
368889999,
377799999,
555999999,
788888888\}.
\end{align*}
Thus, as desired, $x$ is one of the following numbers:
\begin{align*}
    8157999999\overbrace{9\dots9}^{13888881} &= 
    8158000000 \times 10^{13888881}-1,\\
    8368889999\overbrace{9\dots9}^{13888881} &= 8368890000\times 10^{13888881}-1,\\
    8377799999\overbrace{9\dots9}^{13888881} &= 8377800000\times 10^{13888881}-1,\\
    8555999999\overbrace{9\dots9}^{13888881} &= 8556000000\times 10^{13888881}-1, \text{ or }\\
    8788888888\overbrace{9\dots9}^{13888881} &= 8788888889\times 10^{13888881}-1,
\end{align*}
and the smallest of these is $\varepsilon_{13}$.
\end{proof}

\begin{theorem}\label{T:height14}
The smallest elated number of height $14$ is \[\varepsilon_{14} = 8579\times 10^{n_{14}}-1,
\]
where $n_{14} = ((837 \times 10^{13888888} -112)/8-8^2-138)/9^2$.  
Further, $\varepsilon_{14}$ is the only basic elated number of height $14$ less than $8\times 10^{n_{14}+4}$.
\end{theorem}

\begin{proof}
It is straightforward to confirm
that $n_{14}\in \ZZ$.
Let $x$ be a basic elated number of height $14$ such that $x< 8\times 10^{n_{14}+4}$. 
If the leading digit of $x$ is less than $8$, then
\[E_2(x)\leq 7(7^2+(n_{14}+4)9^2) < 8 \times 10^{13888890}
<\varepsilon_{13},\]
which is impossible, since $\varepsilon_{13}$ is minimal of height 13. Thus, the leading digit of $x$ is 8 or 9 and therefore $x$ has at most ${n_{14}+4}$ digits.

Therefore,
\[E_2(x) \leq 9(9^2 + (n_{14}+3)9^2) < 8\times 10^{13888891}\]
and so 
$E_2(x)$ is equivalent to one of the numbers given in Theorem~\ref{T:height13}. None of these numbers is a multiple of 9, 
so the leading digit of $x$ is 8.
It follows that
\[E_2(x) \leq 8(8^2 + (n_{14}+3)9^2) =
728 + 837\times 10^{13888888}.
\]

The fact that $E_2(x)$ is no more than this bound and is equivalent to one of
the numbers listed in 
Theorem~\ref{T:height13} implies that $E_2(x)$ is equivalent to $\varepsilon_{13}$ or $8368890000\times 10^{13888881}-1$. 
But $E_2(x)$ is a multiple of 8 and therefore cannot be equivalent to $\varepsilon_{13}$.  Further there is only one multiple of 8 that is both equivalent to $8368890000\times 10^{13888881}-1$ and no more than $728 + 837\times 10^{13888888}$.   Hence
\[E_2(x) = 836\overbrace{9\dots9}^{13888885}888 = 
837\times 10^{13888888} - 112.\]

Applying Corollary~\ref{S theorem} to $a=E_2(x)/8 - 8^2$, we have
\[q =\left\lfloor \frac{E_2(x)/8-8^2-1}{81}\right\rfloor-6= \left\lfloor n_{14}+\frac{137}{81} \right\rfloor-6 = n_{14}-5\]
and, using Table~\ref{S(a) table},
$\mathscr{S}_{2}(a^\prime) = \mathscr{S}_{2}(543)
= \{57899999\}$.  Thus 
\[x = 857899999\overbrace{9\dots9}^{n_{14}-5} = 
8579 \times 10^{n_{14}} - 1,
\]
completing the proof.
\end{proof}

\begin{theorem}\label{T:height15}
The smallest elated number of height $15$ is \[\varepsilon_{15} = 7489\times 10^{n_{15}}-1,
\]
where $n_{15} = (\varepsilon_{14}/7-7^2-144)/9^2$.  
Further, $\varepsilon_{15}$ is the only basic elated number of height $15$ less than $7\times 10^{n_{15}+4}$.
\end{theorem}

\begin{proof}
We first show that $n_{15} \in \ZZ$.  A direct calculation with a computer algebra package verifies that $n_{14} \equiv 2 \pmod 6$. Therefore, 
\[\varepsilon_{14} 
= 8579\times 10^{n_{14}}-1\equiv
8579 \times 10^{2} - 1 \equiv 0 \pmod 7,\]
implying that $\varepsilon_{14}/7 \in \ZZ.$
Similarly, $n_{14} \equiv 26 \pmod{54}$ and so, applying Euler's Theorem where $\varphi(81) = 54$, 
\[
\varepsilon_{14} \equiv
8579 \times 10^{26} - 1 \equiv 55 \pmod{81}.
\]
Since $7^{-1} \equiv 58 \pmod{81}$, we have $\varepsilon_{14}/7\equiv 55\cdot 58\equiv 31 \pmod{81}$.  Then 
\begin{equation}\label{almostfinalcongruence}
\varepsilon_{14}/7-7^2-144\equiv 0 \pmod{81} 
\end{equation}
implying that
$n_{15} \in \ZZ$, as desired.

Let $x$ be a basic elated number of height 15 such that $ x < 7 \times 10^{n_{15}+4}$.  
If the leading digit of $x$ is less than $7$, then since $x$ has at most $n_{15}+5$ digits, 
\begin{align*}
E_{2}(x)&\leq 6(6^2+(n_{15}+4)9^2)\\
    &= 2160 + 6(\varepsilon_{14}/7-7^2-144)\\
& = 1002 + 6\varepsilon_{14}/7 \\
&< \varepsilon_{14},
\end{align*}
a contradiction.
So the leading digit of $x$ is $7$ or greater and $x$ has at most $n_{15}+4$ digits.

Now
\begin{align*}
E_2(x) & \leq 
9(9^2 + (n_{15}+3)9^2) \\
& = 2916 + 9(\varepsilon_{14}/7 - 7^2 - 144) \\
&= 1179 + 9\varepsilon_{14}/7\\
&=1179+9(8579\cdot 10^{n_{14}}-1)/7 \\
& < 8 \times 10^{n_{14}+4}.
\end{align*}
Thus by 
Theorem~\ref{T:height14}, $E_2(x)$ must be equivalent to $\varepsilon_{14}$.  

Since no number equivalent to $\varepsilon_{14}$ is a multiple of either 8 or 9, the leading digit of $x$ is 7. 
Therefore, 
\begin{align*}
E_2(x)&\leq 7(7^2+(n_{15}+3)9^2)\\
&= 2044 + 7(\varepsilon_{14}/7-7^2-144)\\
& = 693 + \varepsilon_{14}.
\end{align*}
But the only number equivalent to $\varepsilon_{14}$ and at most this bound is $\varepsilon_{14}$ itself.  Thus $E_2(x) = \varepsilon_{14}$. 

Let $a=\varepsilon_{14}/7 - 7^2$. Then, using (\ref{almostfinalcongruence}), $a \equiv 144 \equiv 549\pmod{81}$.  Applying Corollary~\ref{S theorem} and using
Table~\ref{S(a) table}, we have $q = n_{15}-5$ and
$\mathscr{S}_{2}(a^\prime) = \mathscr{S}_{2}(549)
= \{48899999\}$.  Thus 
\[x = 748899999\overbrace{9\dots9}^{n_{15}-5} = 
7489 \times 10^{n_{15}}-1,
\]
completing the proof.
\end{proof}

\begin{theorem}\label{T:height16}
The smallest elated number of height $16$ is \[\varepsilon_{16} = 9189\times 10^{n_{16}}-1,
\]
where $n_{16} = (\varepsilon_{15}/9-9^2-129)/9^2$.  
\end{theorem}

\begin{proof}
To see that $n_{16} \in \ZZ$, we use a computer algebra package to verify the following.  First, 
$10^{1458} \equiv 1 \pmod{3^8\cdot 7}$ and
$n_{14}\equiv 566 \pmod{1458}$ implying that
\[10^{n_{14}}  
\equiv 10^{566} \equiv 4447 \pmod{3^8\cdot 7}.\]
Thus 
\[\varepsilon_{14} = 8579 \times 10^{n_{14}} -1 
\equiv 8579 \cdot 4447 -1 \equiv 31402 \pmod{3^8\cdot 7}.\]
It follows that 
$\varepsilon_{14}/7 \equiv 4486 \pmod{3^8}$, and so
$81n_{15} = \varepsilon_{14}/7 - 7^2 - 144
\equiv 4293 \pmod{3^8}$.  Therefore
$n_{15}\equiv 53 \pmod{3^4}$. 
This combined with the fact that $10^{(3^4)} \equiv 1 \pmod{3^6}$ implies that 
\[\varepsilon_{15} = 7489\times 10^{n_{15}}-1
\equiv 7489 \times 10^{53} - 1 \equiv 432 \pmod{3^6},\]
and thus $\varepsilon_{15}/9 \equiv 48 \pmod{3^4}$.  So, finally, 
\begin{equation}\label{finalcongruence}
\varepsilon_{15}/9-9^2-129 \equiv 0 \pmod{81},
\end{equation}
and hence
$n_{16} \in \ZZ$, as desired.

Let $x$ be a basic elated number of height $16$ such that $x \leq 9189\times 10^{n_{16}}-1$.  
If the leading digit of $x$ is less than 9, then
\begin{align*}
E_2(x)&\leq 8(8^2+(n_{16}+3)9^2)\\
&= 2456 + 8(\varepsilon_{15}/9-9^2-129)\\
&= 776+8\varepsilon_{15}/9 
\\       
&<\varepsilon_{15},
\end{align*}
which is impossible. 
So the leading digit of $x$ is 9. 

Note that 
\begin{align*}
E_2(x)&\leq 9(9^2 + (n_{16}+3)9^2)\\
&=2916 + 9(\varepsilon_{15}/9-9^2-129)\\
& = 1026+ \varepsilon_{15} \\
& < 7\times 10^{n_{15}+4},
\end{align*}
and so, by Theorem~\ref{T:height15},
$E_2(x)$ is equivalent to $\varepsilon_{15}$.
Further, since $E_2(x) \leq \varepsilon_{15} + 1026$,
this implies
that, in fact, $E_2(x) = \varepsilon_{15}$. 

Let $a = \varepsilon_{15}/9 - 9^2$.  
Then, using (\ref{finalcongruence}), $a \equiv 129 \equiv 534\pmod{81}$.  
Applying Corollary~\ref{S theorem} and using
Table~\ref{S(a) table}, we have 
$\mathscr{S}_{2}(a^\prime) = \mathscr{S}_{2}(534) 
= \{18899999,47899999\}$ and $q = n_{16}-5$.  
Thus, recalling that $x \leq 9189\times 10^{n_{16}}-1$, we have
\[x = 918899999\overbrace{9\dots9}^{n_{16}-5} = 
9189 \times 10^{n_{16}}-1,\]
completing the proof.
\end{proof}

\section{Open Problems}\label{openproblems}

We conclude with some open problems related to the work in this paper.  Most of these are derived from questions about happy numbers, some of which are solved in that context and others of which remain open.
\begin{itemize}
\item
Does there exist a base $b > 2$ for which every positive integer is $b$-elated?
\item
Is the number of cycles of $E_{2,b}$ unbounded as the base $b$ increases?
\item
What is the natural density of elated numbers, if it exists?
\item
Does there exist an $N\in \ZZ^+$ such that for each $k > N$, the digits of $\varepsilon_k$ are a permutation of the digits of $E_{2, 10}(\varepsilon_{k+1})$? 
More generally, does such a result hold for any base $b \geq 4$?
\item
Can the results in Section~\ref{heights10} be extended to a practical method for determining 
$\varepsilon_k$ for any $k\geq 17$? 
If so, can the method be extended to other bases?
\end{itemize}

As with the happy function, the elated function can be generalized to exponents other than 2.
For any $e \in \mathbb{Z}^+$, the \emph{$e$-power $b$-elated function}, $E_{e,b}:\mathbb{Z}^+\rightarrow \mathbb{Z}^+$ is defined by
\[
E_{e,b}(a) = E_{e,b}\!\left(\sum_{i=0}^n a_i b^i\right) = a_n\sum_{i=0}^n a_i^e,
\]
where $0\leq a_i <b$ and $a_n\neq 0$.
Any positive integer that maps to 1 under iteration of $E_{e,b}$ is called an \emph{$e$-power $b$-elated number}.  
\begin{itemize}
\item
Which results from this paper generalize directly to higher exponents and which do not?
\item 
What are the best generalizations concerning consecutive or $d$-consecutive sequences of $e$-power $b$-elated numbers for $e > 2$?
What about more generally for $u$-attracted numbers for $E_{e,b}$?
\item
What can be said about the numbers of trailing $(b-1)$'s of the minimal $e$-power $b$-elated numbers of given heights as the heights increase? 
\item
Do the non-leading, non-$(b-1)$ digits of the minimal $e$-power $b$-elated numbers of given heights form a cycle as the heights increase?
\end{itemize}

\section*{Acknowledgements}
The authors thank the organizers of the Research Experiences for Undergraduate Faculty (REUF) 2023. REUF is a program supported by the Institute for Computational and Experimental Research in Mathematics (ICERM) and the American Institute of Mathematics (AIM), and REUF 2023 was supported by NSF grant 2015462 to ICERM.

\end{document}